\documentclass[a4paper,12pt,twoside,notitlepage,reqno]{amsart}

\usepackage{amsmath,amssymb,amsbsy,amsthm, comment}
\usepackage[hmargin=1.4in,vmargin=1.4in]{geometry}
\usepackage{url}

\usepackage{tikz}
\usetikzlibrary{matrix}
\usepackage[linktocpage]{hyperref}
\hypersetup{
    colorlinks=true,        
    linkcolor=red,          
    citecolor=red,         
    filecolor=magenta,      
    urlcolor=cyan           
}

\addtocontents{toc}{\protect\setcounter{tocdepth}{1}}

\usepackage{eucal}
\usepackage[all]{xy}

\newcommand{\C}{ \mathbb{C} }

\newcommand{\Q}{ \mathbb{Q} }
\newcommand{\Z}{ \mathbb{Z} }
\newcommand{\N}{ \mathbb{N} }
\newcommand{\F}{ \mathbb{F} }

\newcommand{\dra}{\dashrightarrow}
\newcommand{\lra}{\longrightarrow}

\newcommand{\bG}{\mathbb{G}}

\newcommand{\bZ}{\mathbb{Z}}

\theoremstyle{plain}
	\newtheorem{thm}{Theorem}
	
		\numberwithin{thm}{section}
	\newtheorem{lemma}[thm]{Lemma}
	\newtheorem{prop}[thm]{Proposition}
	\newtheorem{cor}[thm]{Corollary}

	\newtheorem*{thm*}{Theorem}
	\newtheorem*{lemma*}{Lemma}
	\newtheorem*{prop*}{Proposition}
	\newtheorem*{cor*}{Corollary}
	\newtheorem*{conj*}{Conjecture}

\theoremstyle{definition}
	\newtheorem{example}[thm]{Example}
	\newtheorem*{example*}{Example}
	\newtheorem{defn}[thm]{Definition}
	\newtheorem{remark}[thm]{Remark}

\makeatletter
\@namedef{subjclassname@2020}{%
  \textup{2020} Mathematics Subject Classification}
\makeatother

\begin{document}

\title[Orbit intersections in semiabelian varieties]{Intersections of orbits of self-maps with subgroups in semiabelian varieties}

\author{Jason Bell}
\address{University of Waterloo \\
Department of Pure Mathematics \\
Waterloo, Ontario \\
N2L 3G1, Canada}
\email{jpbell@uwaterloo.ca}

\author{Dragos Ghioca}
\address{University of British Columbia\\
Department of Mathematics\\
Vancouver, BC\\
V6T 1Z2, Canada}
\email{dghioca@math.ubc.ca}


\thanks{The authors were partially supported by Discovery Grants from the National Science and Engineering Research Council of Canada.}

\begin{abstract}
Let $G$ be a semiabelian variety defined over an algebraically closed field $K$,  endowed with a rational self-map $\Phi$. Let $\alpha\in G(K)$ and let $\Gamma\subseteq G(K)$ be a finitely generated subgroup. We show that the set $\{n\in\N\colon \Phi^n(\alpha)\in \Gamma\}$ is a union of finitely many arithmetic progressions along with a set of Banach density equal to $0$. In addition, assuming $\Phi$ is regular, we prove that the set $S$ must be finite.
\end{abstract}

\maketitle



\section{Introduction}

The Mordell-Lang conjecture, now a theorem due to Faltings \cite{Fal83}, asserts that a subvariety of an abelian variety $A$ defined over a field of characteristic zero intersects a finitely generated subgroup $\Gamma$ of $A$ in a finite union of cosets of subgroups $\Gamma$.  This was later extended by Vojta \cite{Vojta1,Vojta2} (see also \cite{McQ}), who showed the analogous result holds for semiabelian varieties; that is, (commutative) algebraic groups $G$ that lie in a short exact sequence of algebraic groups
$$1\to \mathbb{G}_m^N \to G \to A\to 1$$ with $A$ an abelian variety and $N$ a nonnegative integer.
This result is noteworthy in that it shows that the interaction between the geometric structure with the underlying group theoretic structure of a semiabelian variety is well-behaved.  In recent years, the Mordell-Lang conjecture has inspired the so-called Dynamical Mordell-Lang conjecture, in which one now has an algebraic variety with a rational self-map and one now seeks to show that the interaction between the geometric structure and the dynamical structure is again well-behaved.  More precisely, the Dynamical Mordell-Lang conjecture asserts that if $X$ is a quasi-projective variety defined over a field of characteristic zero, and $\Phi: X\dra X$ is a rational self-map and $c$ and $Y$ are respectively a closed point of $X$ whose forward orbit under $\Phi$ is well-defined and a Zariski closed subset of $X$, then the set $\{n\in \mathbb{N}_0\colon \Phi^n(c)\in Y\}$ is a finite (possibly empty) union of infinite arithmetic progressions along with a (possibly empty) finite set. (We recall that for a rational self-map $\Phi$ on a variety $X$ and for $n\in\N_0$, $\Phi^n$ denotes the $n$-th iterate of $\Phi$, where $\Phi^0$ is taken to be the identity map by convention.). For more details regarding the Dynamical Mordell-Lang conjecture, we refer the reader to the book~\cite{DML-book}.

If, however, one returns to the setting of semiabelian varieties, which gave impetus to the 
Dynamical Mordell-Lang conjecture, it is very natural to ask whether the interaction between the 
dynamical structure and the group theoretic structure is similarly well-behaved when one 
has a rational self-map $\Phi$ of a semiabelian variety $G$.  In this paper we study this 
question and show that in the case that $\Phi$ is a morphism one in fact obtains the 
analogous conclusion as in the statement of the Dynamical Mordell-Lang conjecture.  In the 
case when $\Phi$ is a rational self-map the conclusion does not hold in general, but a 
weaker version holds in which the finite set is replaced by a set of zero Banach density.  
Intuitively, a set of zero Banach density is a \emph{very sparse} set.   Precisely, given a 
subset $I\subseteq\N_0$, we define the (upper) Banach density, $\bar{\delta}(I)$, of $I$ 
using the formula 
\begin{equation}
\label{eq:Banach}
\bar{\delta}(I):=\limsup_{|J|\to\infty}\frac{|I\cap J|}{|J|},
\end{equation}
where the above $\limsup$ is computed with respect to finite intervals $J$ in the natural numbers. We note that due to definition~\eqref{eq:Banach}, one could have potentially a set of natural density $0$, but of Banach density equal to $1$, and so the condition that a set have Banach density zero is a much stronger constraint than merely being of zero natural density.

Our main result is given by the following theorem.
\begin{thm}
\label{thm:main}
Let $G$ be a semiabelian variety defined over an algebraically closed field $K$ endowed with a rational self-map $\Phi$, let $\alpha\in G(K)$ be a point for which its orbit under the action of $\Phi$ is well-defined, and let $\Gamma\subset G(K)$ be a finitely generated subgroup.
Then the following hold:
\begin{enumerate}
\item[(i)] the set $\{n\colon \Phi^n(\alpha)\in \Gamma\}$ is a finite union of arithmetic progressions along with a set of zero Banach density.
\item[(ii)]  if, in addition, $\Phi$ is regular then $\{n\colon \Phi^n(\alpha)\in \Gamma\}$ is a finite union of arithmetic progressions along with a finite set.
\end{enumerate}
\end{thm}

\begin{remark}
\label{rem:important}
The conclusions to the statement of Theorem \ref{thm:main} (i) and (ii) both hold if we replace $\Gamma$ by a coset of a finite of a finitely generated group, which can be seen by conjugating $\Phi$ by a suitable translation map and suitably modifying the point $\alpha$.
\end{remark}

Theorem~\ref{thm:main} (ii), while not obviously connected to the Dynamical Mordell-Lang conjecture, in fact quickly implies that the Dynamical Mordell-Lang conjecture holds for dynamical systems $(G,\Phi)$ when $G$ is a semiabelian variety and $\Phi$ is a regular self-map, if one also uses the famous theorem of Vojta \cite{Vojta1, Vojta2} proving the classical Modell-Lang Conjecture (see Corollary \ref{cor:1}). The Dynamical Mordell-Lang conjecture for semiabelian varieties was proven in \cite{GT0} for regular self-maps on semiabelian varieties, but the method employed in \cite{GT0} does not use the above route and instead employs a $p$-adic approach which is (along with the main result of \cite{Jason}) the precursor of the so-called \emph{$p$-adic arc lemma} from \cite{BGT0}. 

Our Theorem~\ref{thm:main} is also connected to the main result of \cite{BGT-approximate} in which it was shown that for any algebraic dynamical system $(X,\Phi)$, where $X$ is a variety and $\Phi$ is even a rational self-map, then given $x\in X$, the set of $n$ for which $\Phi^n(x)$ lies in a fixed subvariety $Y$ of $X$ is a union of finitely many arithmetic progressions along with a set of Banach density $0$. However, the method of proof from \cite{BGT-approximate} employs Noetherian induction and it cannot be modified to prove the conclusion from Theorem~\ref{thm:main}, which has a more algebraic flavour in the spirit of the classical Mordell-Lang conjecture.

Our Theorem~\ref{thm:main} is connected to a result of the authors regarding a fusion variant of the classical and dynamical Mordell-Lang conjectures, which itself built on the paper \cite{BCE}. More precisely, in \cite{BCE, BG-2}, it is shown that for a dominant rational self-map $\Phi$ on a variety $X$ defined over a field $K$, endowed with a rational function $f:X\dashrightarrow \mathbb{P}^1$, then for a starting point $\alpha\in X(K)$ with a well-defined orbit under $\Phi$ and for a finitely generated subgroup $\Gamma\subset \bG_m(K)$, the set of all $n\in\N_0$ for which $f(\Phi^n(\alpha))\in\Gamma$ is a union of finitely many arithmetic progressions along with a set of Banach density equal to $0$.  This result plays a crucial role in our proof of Theorem \ref{thm:main} (i).

Finally, we note that in the conclusion to the statement of Theorem~\ref{thm:main} (i), one cannot in general replace the set of zero Banach density with a finite set, as the following examples show.

\begin{example}
\label{ex:1}
For the rational self-map $\Phi$ on $\bG_m$ given by $x\mapsto x+1$, 
along with the starting point $\alpha:=1$ and the subgroup $\Gamma\subset \bG_m(\mathbb{Q})$ generated by $2$, we see
$$\{n\colon \Phi^n(\alpha)\in \Gamma\}=\{0\}\cup\left\{2^n\colon n\ge 0\right\}.$$
\end{example}

Similar examples arise in positive characteristic, such as the following one.
\begin{example}
\label{ex:2}
For the map $\Phi:\bG_m\dra \bG_m$ defined over $\F_p(t)$ by the formula:  $\Phi(x)=tx-t+1$, we see that $\Phi^n(t+1)=t^{n+1}+1$ for each $n\ge 0$. So for the cyclic subgroup $\Gamma\subset\bG_m(\F_p(t))$ generated by $t+1$ and $\alpha=t+1$, we have
$$\{n\colon \Phi^n(\alpha)\in \Gamma\} =\left\{p^n-1\colon n\in\N_0\right\}.$$
\end{example}

It is tempting to conjecture that the sparse sets $R(G,\Phi,\alpha,\Gamma)$ from the conclusion of Theorem~\ref{thm:main} are always the image of some exponential function such as in Examples~\ref{ex:1}~and~\ref{ex:2}. However, as shown in \cite[Section~4]{JIMJ}, it is very hard to predict the exact shape of return sets associated with questions having the flavour of the Dynamical Mordell-Lang problem in characteristic $p$.  


The outline of this paper is as follows. We begin by presenting various technical results in Section~\ref{sec:technical} regarding linear recurrence sequences, semiabelian varieties and also basic theory of finitely generated groups, which will later be employed in our proof of Theorem~\ref{thm:main}. Then we prove Theorem~\ref{thm:main} (ii) in Section~\ref{sec:special}. We conclude with our proof of Theorem~\ref{thm:main} (i) in Section~\ref{sec:general}. 



\section{Technical background}
\label{sec:technical}

We collect here the various technical results later employed in our proofs.


\subsection{Linear recurrence sequences}
\label{subsec:lrs}

In this section we state the Skolem-Mahler-Lech theorem which will be
used in our proof. First we need to
introduce the basic set-up for \emph{linear recurrence sequences} (see
\cite{Schinzel} for more details on linear recurrent sequences).
\begin{defn}
\label{recurrence sequences}
Let $(H,+)$ be an abelian group. The sequence $\{u_n\}_{n\in\N_0}\subseteq H$ is a linear recurrence sequence (defined over the integers), if there exists a positive integer $k$ and there exist constants $c_1,\dots,c_{k}\in\Z$  such that
\begin{equation}
\label{the relation}
u_{n+k} = \sum_{i=1}^{k} c_i u_{n+k-i}\text{, for each $n\in\N_0$.}
\end{equation}
\end{defn}

The following result is the well-known Skolem-Mahler-Lech theorem which applies to general linear recurrence sequences of complex numbers (not necessarily defined over the integers); for more details, see \cite{Skolem, Mahler, Lech}.

\begin{prop}
\label{ever 3}
Let $\{u_k\}_{k\in\N_0}\subset \C$ be a linear recurrence sequence and let $C\in\C$. Then the set $\{n\in\N_0\colon u_n=C\}$ 
is a union of finitely many arithmetic progressions along with a finite set.
\end{prop}


\subsection{Semiabelian varieties}
\label{subsec:technical_semiabelian}

For a semiabelian variety $G$ defined over an algebraically closed field $K$ there exists a short exact sequence of algebraic groups defined over $K$: 
\begin{equation}
\label{eq:short_exact_0}
1\lra \bG_m^N\lra G\lra A\lra 1,
\end{equation}
where $N\in\N_0$ and $A$ is an abelian variety. Furthermore, each regular self-map $\Phi$ of $G$ is a composition of a translation with a group endomorphism. Finally, for each group endomorphism $\Psi$ of $G$, there exist integers $a_0,a_1,\dots, a_{g-1}$ (where $g\le 2\dim(G)$) such that 
\begin{equation}
\label{eq:integral_0}
\Psi^g=\sum_{i=1}^{g}a_i\Psi^{g-i}.
\end{equation}
For more details regarding semiabelian varieties, we refer the reader to \cite{Iitaka, NW}.


\subsection{Finitely generated subgroups}
\label{subsec:f.g.}

We conclude by deriving a useful result regarding finitely generated subgroups  which will be employed in our proof of Theorem~\ref{thm:main}. First we need a definition.
\begin{defn}
\label{def:f.g.}
Let $\Gamma$ be a finitely generated subgroup of an abelian group $(G,+)$. We say that the elements $z_1,\dots, z_n\in G$ are linearly independent with respect to $\Gamma$ if for each integers $k_1,\dots, k_n$, we have that
$$\sum_{i=1}^n k_iz_i\in\Gamma$$
if and only if $k_1=\cdots = k_n=0$.
\end{defn}

\begin{prop}
\label{prop:f.g.}
Let $\Gamma$ be a finitely generated subgroup of the abelian group $(G,+)$. Let $m\in\N$ and $y_1,\dots, y_n,z_1,\dots, z_n\in G$ with the property that  $z_1,\dots, z_n$ are linearly independent with respect to the subgroup $\Gamma_1$ spanned by $\Gamma$ along with $y_1,\dots, y_n$. Let  $y_i':=y_i+z_i$ for each $i=1,\dots, n$, and we let the subgroup $\Gamma_1'\subseteq G$ spanned by $\Gamma$ along with $y_1',\dots, y_n'$. Then $\Gamma_1\cap \Gamma_1'=\Gamma$. 
\end{prop}

\begin{proof}
Clearly, $\Gamma$ is contained in $\Gamma_1\cap \Gamma_1'$, so it suffices to show that $\Gamma_1\cap \Gamma_1'\subseteq \Gamma$.  Let $\gamma\in \Gamma_1\cap \Gamma_1'$.  Then there exist integers  $k_1,\dots, k_n$ and $k_1',\ldots ,k_n'$ such that
$\gamma\in k_1 y_1+\cdots + k_n y_n + \Gamma$ and $\gamma\in k_1' y_1' + \cdots + k_n' y_n' +\Gamma$.  Hence $$k_1 y_1+\cdots + k_n y_n - k_1' y_1'-\cdots -k_n' y_n' \in \Gamma.$$  But this gives that $-k_1' z_1 -\cdots -k_n' z_n \in \Gamma_1$ and since $z_1,\ldots ,z_n$ are linearly independent with respect to the subgroup $\Gamma_1$, we must have $k_1'=\cdots =k_n'=0$ and so $\gamma\in \Gamma$ as desired.
The result follows.
\end{proof}


\section{Proof of Theorem~\ref{thm:main} (ii)}
\label{sec:special}

For this case, the strategy is inspired by the proofs from \cite{TAMS, JIMJ}, even though our arguments are simpler. 


We begin with a useful result about orbits of points in semiabelian varieties.
\begin{lemma}
Let $G$ be a semiabelian variety and let $\Phi: G\to G$ be a regular self-map of $G$.  If $\alpha\in G$ then $\{\Phi^n(\alpha)\colon n\ge 0\}$ is contained in a finitely generated subgroup of $G$.
\label{lem:fg}
\end{lemma}
\begin{proof}
The regular self-map $\Phi:G\lra G$ can be written as $T_\beta\circ \Psi$, where $T_\beta$ is the translation-by-$\beta$ map on $G$ and $\Psi$ is a group endomorphism.  Since every endomorphism of a semiabelian variety is integral over  $\bZ$ (see \eqref{eq:integral_0}), we may let $X^g - \sum_{i=1}^g e_i X^{g-i}$ (for some $g\in\N$) be the minimal
  polynomial of $\Psi$ over $\Z$. Then, for each $k\ge 0$, we
  have that for a point $\gamma\in G$: 
\begin{equation}
\label{recurrence 2-0}
\Psi^{n+g}(\gamma) = \sum_{i=1}^g e_i \Psi^{n+g-i}(\gamma),
\end{equation}
for each $n\in\N_0$.
Furthermore, given $\alpha\in G(K)$, we have the general formula:
\begin{equation}
\label{eq:general_formula_0}
\Phi^n(\alpha)=\Psi^n(\alpha)+\sum_{j=0}^{n-1}\Psi^j(\beta).
\end{equation}
Employing equation~\eqref{eq:general_formula_0} and the recurrence relation~\eqref{recurrence 2-0} applied to the points $\alpha$ and $\beta$, we see that 
\begin{equation}
\label{recurrence 2}
\Phi^{n+g+1}(\alpha) = (e_1+1)\Phi^{n+g}(\alpha)+\sum_{i=2}^{g} (e_i-e_{i-1})\Phi^{n+1+g-i}(\alpha) - e_g\Phi^n(\alpha).
\end{equation}
So, because $\{\Phi^n(\alpha)\}_{n\in\N_0}$ is a linear recurrence sequence of order $g+1$, then the orbit of $\alpha$ under $\Phi$ is contained in $\Gamma_1$, where $\Gamma_1$ is the finitely generated subgroup of $G$ spanned by  the generators of $\Gamma$, along with  $\Phi^i(\alpha)$  for $i=0,\dots,g$.  
\end{proof}
Hence, it suffices to show the following general statement.
\begin{prop}
\label{SML fin gen}
Let $\Gamma_1$ be a finitely generated group, and let $\Gamma$ be a subgroup of $\Gamma_1$. Then for every linear recurrence sequence $\{x_n\}_{n\in\N_0}\subset\Gamma_1$ defined over the integers, the set of all $n\in\N_0$ such that $x_n\in\Gamma$ is a union of finitely many arithmetic progressions along with a finite set. 
\end{prop} 
\begin{proof} We first note that it suffices to prove the proposition in the case that $\Gamma=\{0\}$. Indeed, we let $\pi:\Gamma_1\lra \Gamma_1/\Gamma$ be the canonical projection. Then $y_n:=\pi(x_n)$, for each $n\in\N$, forms another linear recurrence sequence, also defined over the integers. Hence, each $n\in\N$ satisfies $x_n\in\Gamma$ if and only if $y_n=0$.

Thus we may assume that $\Gamma=\{0\}$ for the remainder of the proof and so we only need to show that the set of all integers $n$ such that $x_n=0$ consists of at most finitely many arithmetic progressions in $\N$ along with a finite set. 

Using the fact that the intersection of two arithmetic progressions is another arithmetic progression (or the empty set), we see that if we can write $\Gamma_1=H_1\oplus H_2$, and let $\pi_1:\Gamma_1\lra H_1$ and $\pi_2:\Gamma_1\lra H_2$ be the canonical projections, then for $n\in\N_0$, let  $y_n:=\pi_1(x_n)$ and $z_n:=\pi_2(x_n)$. It suffices to prove that the set of $n$ for which $y_n\in H_1$ and the set of $n$ for which $z_n\in H_2$ are both expressible as a finite union of arithmetic progressions along with a finite set.

Then since $\Gamma_1$ can be written as a finite direct sum of cyclic groups, we see it suffices to prove Proposition~\ref{SML fin gen} when $\Gamma_1$ is cyclic and $\Gamma=\{0\}$. Now if $\Gamma_1$ is an infinite cyclic group, then $\Gamma_1$ is isomorphic to $\bZ$ (as a group) and the desired conclusion follows from Proposition~\ref{ever 3}. Thus we may assume that
$\Gamma_1$ is finite and so $\Gamma_1\cong \bZ/N\bZ$, for some positive integer $N$. Because $\{x_n\}_{n\in\N_0}$ is a linear recurrence sequence (over $\bZ$) contained in a finite group, we conclude that $\{x_n\}_{n\in\N_0}$ is preperiodic (i.e., there exist integers $k\ge 0$ and $\ell\ge 1$ such that $x_{n+\ell}=x_n$ for each $n\ge k$). Thus the set of all $n\in\N$ such that $x_n=0$ consists of at most finitely many arithmetic progressions along with a finite set.
The result follows.
\end{proof}
\begin{proof}[Proof of Theorem \ref{thm:main} (ii)]
By Lemma \ref{lem:fg}, we have $\{\Phi^n(\alpha)\colon n\ge 0\}$ is contained in finitely generated group. Hence there exists a finitely generated subgroup $\Gamma_1$ of $G$ that contains both $\Gamma$ and $\Phi^j(\alpha)$ for $j\ge 0$.  Then $x_n:=\Phi^n(\alpha)$ lies in $\Gamma_1$ for each $n\ge 0$ and the sequence $\{x_n\}$ satisfies a linear recurrence as the arguments in Lemma \ref{lem:fg} and Equation (\ref{recurrence 2}) show.  Then applying Proposition \ref{SML fin gen} gives the desired result.
\end{proof}
As an immediate consequence, we obtain the Dynamical Mordell-Lang conjecture or semiabelian varieties, which was proved by different methods in \cite{GT0}.
\begin{cor}
Let $\Phi$ be a regular self-map of a semiabelian variety $G$.  Then the Dynamical Mordell-Lang conjecture holds for the dynamical system $(G,\Phi)$.
\label{cor:1}
\end{cor}
\begin{proof} By Lemma \ref{lem:fg} we have that the orbit of $c$ under $\Phi$ is contained in a finitely generated subgroup $\Gamma$ of $G$.  Then by Vojta's theorem \cite{Vojta1,Vojta2}, $\Gamma\cap Y$ is a finite union of cosets of subgroups of $\Gamma$, say $\{\gamma_i+\Gamma_i\colon i=1,\ldots ,m\}$; by construction $\Phi^n(c)\in Y$ if and only if $\Phi^n(c)\in \gamma_i+\Gamma_i$ for some $i$. Then by Theorem \ref{thm:main} (ii) and Remark \ref{rem:important} we have that the set of $n$ for which $\Phi^n(c)\in \gamma_i+\Gamma_i$ is a finite union of arithmetic progressions along with a finite set for $i=1,\ldots ,m$ and since such sets are closed under the process of taking finite unions, we obtain the desired result.  
\end{proof}


\section{Proof of Theorem~\ref{thm:main} (i)}
\label{sec:general}
Throughout this section we let $G$ be a semiabelian variety  and let $\Gamma$ be a finitely generated subgroup of $G$.  Given a rational self-map $\Psi$ of $G$ and $\beta\in G$ whose forward orbit under $\Psi$ is defined, we define the \emph{return set}
\begin{equation}
\label{eq:return_set}
R(G,\Psi,\beta,\Gamma):=\left\{n\in\N_0\colon \Psi^n(\beta)\in\Gamma\right\}.
\end{equation}
We proceed first with a useful reduction for our proof.

\begin{remark}
\label{rem:k-l}
Throughout our proof of Theorem~\ref{thm:main} we will often find useful to replace $(\alpha,\Phi)$ by $\left(\Phi^\ell(\alpha),\Phi^k\right)$. Indeed, given $k\in\N$, once we prove that for each $\ell=0,\dots, k-1$, the sets
$$R(G,\Phi^k,\Phi^\ell(\alpha),\Gamma):=\left\{n\in\N_0\colon \Phi^{kn}(\Phi^\ell(\alpha))\in \Gamma\right\}$$
is a union of finitely many arithmetic progressions along with a set of Banach density equal to $0$, then clearly also $R(G,\Phi,\alpha,\Gamma)$ is a union of finitely many arithmetic progressions along with a set of Banach density equal to $0$.
\end{remark}


\subsection{General setup}
\label{subsec:general}

We have a rational self-map $\Phi$ on a semiabelian variety $G$ defined over an algebraically closed field $K$. In particular, there exists a short exact sequence of algebraic groups (see~\eqref{eq:short_exact_0})
\begin{equation}
\label{eq:short_sequence}
1\lra \bG_m^N\lra G\lra A\lra 1, 
\end{equation}
where $A$ is an abelian variety; we denote by $\pi:G\lra A$ the projection from \eqref{eq:short_sequence}. Also, if $\pi$ were an isomorphism (i.e., $N=0$ in \eqref{eq:short_sequence}), then the rational self-map $\Phi$ would actually be regular and so, we would be done by Theorem~\ref{thm:main} (ii) proven in Section~\ref{sec:special}. Thus we may assume henceforth that $N$ is a positive integer (and also identify the maximal algebraic torus inside $G$ with $\bG_m^N$). 


Now, because the only rational maps $\mathbb{A}^1\dashrightarrow A$ are constant (since $A$ is an abelian variety), we see that $\Phi$ induces a rational self-map $\bar{\Phi}$ on $A$, i.e., $\pi\circ \Phi = \bar{\Phi}\circ \pi$. 
On the other hand, since any rational self-map on an abelian variety is regular, we get that $\bar{\Phi}$ is a regular self-map on $A$.


\subsection{A linear recurrence relation}
\label{subsec:lrs_abelian}

For the  regular self-map $\bar{\Phi}$ of the abelian variety $A$, arguing identically as in Section \ref{sec:special} (see equation~\eqref{recurrence 2}), we obtain that there exists a positive integer $m\le 2\dim(A)+1$ with the property that there exist integers $c_1,\dots,c_m$  such that
\begin{equation}
\label{eq:recurrence_22}
\bar{\Phi}^{m}=\sum_{j=1}^m c_j\bar{\Phi}^{m-j}. 
\end{equation} 
We derive a relation such as \eqref{eq:recurrence_22} by noting that $\bar{\Phi}$ is the composition of a translation with a group endomorphism of $A$. For any group endomorphism of $A$ there exists a recurrence relation of the form~\eqref{eq:recurrence_22} with $m\le 2\dim(A)$. Then, arguing as in the transition from Equation~\eqref{recurrence 2-0} to Equation~\eqref{recurrence 2}, we obtain a recurrence relation of the form~\eqref{eq:recurrence_22} for $\bar{\Phi}$ with $m\le 1+2\dim(A)$.

We define a group endomorphism $\Psi:G^m\lra G^m$ as follows
\begin{equation}
\label{eq:def_psi}
\Psi(y_1,\dots, y_m)=\left(y_2,\dots, y_m,\sum_{i=1}^m c_jy_{m+1-j}\right).
\end{equation}
We also let $\tilde{\pi}:G^m\lra A^m$ be defined as 
$$\tilde{\pi}(y_1,\dots, y_m)=\left(\pi(y_1),\dots, \pi(y_m)\right).$$
Then letting $\bar{x}:=\pi(x)$ for each $x\in G(K)$, and also using Equations~\eqref{eq:def_psi} and \eqref{eq:recurrence_22}, we obtain that for each $n\in\N_0$, we have that 
\begin{equation}
\label{eq:prop_psi}
\left(\tilde{\pi}\circ \Psi^n\right)\left(\alpha, \Phi(\alpha),\dots, \Phi^{m-1}(\alpha)\right) = \left(\bar{\Phi}^n(\bar{\alpha}),\bar{\Phi}^{n+1}(\bar{\alpha}),\dots, \bar{\Phi}^{n+m-1}(\bar{\alpha})\right). 
\end{equation}
For the sake of simplifying our notation later, we denote:
$$\bar{\alpha}_i:=\bar{\Phi}^i(\bar{\alpha})\text{ for }i=0,\dots, m-1.$$


\subsection{Another finitely generated subgroup}
\label{subsec:another_f.g.}

We let $\beta_0,\dots, \beta_{m-1}\in G(K)$ such that 
\begin{equation}
\label{eq:y_i}
\pi(\beta_i)=\bar{\alpha}_i\text{ for }i=0,\dots, m-1.
\end{equation}
We let $\Gamma_1$ be the finitely generated subgroup of $G(K)$ spanned by $\Gamma$ along with $\beta_0,\dots, \beta_{m-1}$.
\begin{remark}
\label{rem:y}
The subgroup $\Gamma_1$ depends on the choice for $\beta_0,\dots, \beta_{m-1}$ verifying equation~\eqref{eq:y_i} (one obvious choice is $\beta_i=\Phi^i(\alpha)$ for $i=0,\dots, m-1$, but there are infinitely many possibilities for the $\beta_i$'s since $N\ge 1$ and thus, $\pi:G\lra A$ has infinite kernel). We will prove that (regardless of the choice for the $\beta_i$'s) the set 
\begin{equation}
\label{eq:R_1}
R_1:=R(G,\Phi,\alpha,\Gamma_1):=\left\{n\in\N_0\colon \Phi^n(\alpha)\in \Gamma_1\right\}
\end{equation}
is a union of finitely many arithmetic progressions along with a set of Banach density equal to $0$. Then working with two different liftings $\beta_0,\dots, \beta_{m-1}$ and $\beta_0',\dots, \beta_{m-1}'$ of $\bar{\alpha}_0,\dots,\bar{\alpha}_{m-1}$ (which satisfy the hypotheses of Proposition~\ref{prop:f.g.}) which generate two subgroups $\Gamma_1$ and $\Gamma_1'$ for which $\Gamma=\Gamma_1\cap\Gamma_1'$ will allow us to derive the desired conclusion in Theorem~\ref{thm:main} (i).
\end{remark}

In Section~\ref{subsec:R_1_reformulated} we reformulate the desired conclusion for the set $R_1$ from Equation~\eqref{eq:R_1} in terms of a new finitely generated subgroup of $\bG_m^N$.


\subsection{Reformulating the desired conclusion for our new finitely generated subgroup}
\label{subsec:R_1_reformulated}

We consider the following rational self-map $\tilde{\Phi}:G^{m+1}\dra G^{m+1}$ given by
\begin{equation}
\label{eq:def_tilde_phi}
\tilde{\Phi}(x,y_1,\dots, y_m):=\left(\Phi(x),\Psi(y_1,\dots, y_m)\right).
\end{equation}
Then using Equation~\eqref{eq:prop_psi} we see that for each $n\in\N_0$, we have that 
\begin{equation}
\label{eq:def_beta_n}
\tilde{\Phi}(\alpha,\beta_0,\dots, \beta_{m-1}):=\left(\Phi^n(\alpha),\beta_n,\dots, \beta_{n+m-1}\right)
\end{equation}
for some points $\beta_n\in G(K)$, where (very importantly) due to Equations \eqref{eq:prop_psi} and \eqref{eq:y_i}, we have that for each $n\in\N_0$:  
\begin{equation}
\label{eq:prop_beta_n}
\pi\left(\Phi^n(\alpha)\right)=\pi(\beta_n). 
\end{equation}
Letting $\Theta:G^{m+1}\lra G^{m+1}$ be the group endomorphism defined by the rule
\begin{equation}
\label{eq:def_theta}
\Theta(x,y_1,\dots, y_m):=(x-y_1,y_1,\dots, y_m),
\end{equation}
we have that for each $n\in\N_0$,
\begin{equation}
\label{eq:prop_theta}
\left(\Theta\circ \tilde{\Phi}^n\right)\left(\alpha,\beta_0,\dots, \beta_{m-1}\right)\in \bG_m^N(K)\times G^m(K).
\end{equation}

We let $Y\subset G^{m+1}$ be the Zariski closure of the orbit of $\tilde{\alpha}:=(\alpha,\beta_0,\dots, \beta_{m-1})$ under the action of $\tilde{\Phi}$; clearly, $\tilde{\Phi}(Y)\subseteq Y$. For the sake of not complicating the notation, we let $\tilde{\Phi}$ also denote the induced rational self-map on $Y$. 

Using Equation~\eqref{eq:prop_theta} along with the fact that $Y$ is the Zariski closure of the orbit of $\alpha$ under $\tilde{\Phi}$ on which the property~\eqref{eq:prop_theta} holds, we conclude that $\Theta(Y)\subseteq \bG_m^N\times G^m$.  Therefore, for each $i=1,\dots, N$, we have well-defined projection maps $\pi_i:\Theta(Y)\dra \bG_m$ onto each one of the $N$ coordinates of $\bG_m^N$.

We now consider the intersection 
\begin{equation}
\label{eq:coset_G_m}
H:=\Gamma_1\cap \bG_m^N(K),
\end{equation}
and let $\tilde{\pi}:\bG_m^N\times G^m\lra \bG_m^N$ be the natural projection. 

\begin{lemma}
\label{lem:the_same}
Adopt the notation above.  Then $R_1=R(G,\Phi,\alpha,\Gamma_1)$ is equal to the set 
\begin{equation}
\label{eq:new_return}
\tilde{R}_1:=\left\{n\in\N_0\colon \left(\tilde{\pi}\circ \Theta\circ \tilde{\Phi}^n\right)\left(\alpha, \beta_0,\dots, \beta_{m-1}\right)\in H\right\}.
\end{equation}
\end{lemma}

\begin{proof}
We recall that $\Psi^n(\beta_0,\dots, \beta_{m-1})=(\beta_n,\dots, \beta_{n+m-1})$ for each $n\ge 0$ and so from the definition of $\Psi$ we have that $\beta_n$ is contained in the linear span of $\beta_0,\dots, \beta_{m-1}$.  Hence $\beta_n\in\Gamma_1$ for all $n$. Therefore, we have that $\Phi^n(\alpha)\in \Gamma_1$ if and only if $(\Phi^n(\alpha)-\beta_n)\in\Gamma_1$. Furthermore, since $\Phi^n(\alpha)-\beta_n\in \bG_m^N(K)$ (see equation~\eqref{eq:prop_beta_n}), then $(\Phi^n(\alpha)-\beta_n)\in \Gamma_1$ if and only if $\Phi^n(\alpha)-\beta_n\in H$, as desired.
\end{proof}

In Section~\ref{subsec:in_R_1} we show that the set $R_1$ from Equation~\eqref{eq:R_1} (which is the same as the set $\tilde{R}_1$ from \eqref{eq:new_return}, as shown in Lemma~\ref{lem:the_same}) is indeed a union of finitely many arithmetic progressions along with a set of Banach density equal to $0$.


\subsection{Deriving the conclusion for the new finitely generated subgroup}
\label{subsec:in_R_1}

Throughout this section we let $H$ be as in Equation~\eqref{eq:coset_G_m}.
Then there exists a finitely generated subgroup $E\subset \bG_m(K)$ such that $H\subseteq E^N$ (we can, for example, take $E$ to be the subgroup of $K^*$ spanned by all the entries of the generators of $H$). 

Next, for each $i=1,\dots, N$, we let $f_i:Y\dashrightarrow \bG_m$ be the rational map
\begin{equation}
\label{eq:f_i}
f_i:=\pi_i\circ \Theta.
\end{equation} 
Now, \cite[Theorem~1.1]{BCE} and \cite[Theorem~1.1]{BG-2} yield that for each $i=1,\dots, N$, the set 
\begin{equation}
\label{eq:R_i}
U_i:=\left\{n\in\N_0\colon f_i\left(\tilde{\Phi}^n\left(\tilde{\alpha}\right)\right)\in E\right\}
\end{equation}
is a union of finitely many arithmetic progressions along with a set of Banach density equal to $0$. Therefore, using that intersections of finite unions of arithmetic progressions is also a finite union of arithmetic progressions, along with the fact that $H\subseteq E^N$, then we see that 
\begin{equation}
\label{eq:in_intersection}
R_1\subseteq \tilde{U}:=\bigcap_{i=1}^N U_i 
\end{equation}
and moreover, $\tilde{U}$ is  a union of finitely many arithmetic progressions along with a set of Banach density equal to $0$. 

\begin{remark}
\label{rem:general_variety}
The variety $Y$ is not necessarily irreducible. However, it has finitely many irreducible components $Y_j$ and an iterate $\tilde{\Phi}^\ell$ of $\tilde{\Phi}$ induces rational self-maps on each irreducible component of $Y$. Hence the result of \cite{BG-2} can be applied to each $(Y_j,\tilde{\Phi}^\ell)$, which still allows us to derive the desired conclusion about each $U_i$ (see also Remark~\ref{rem:k-l}). Also, note that the results of \cite{BCE} (which were then extended in \cite{BG-2} for fields of arbitrary characteristic) are written for arbitrary varieties, not necessarily irreducible. Therefore, even when $Y$ is not irreducible, we still obtain the desired description of each $U_i$ from Equation~\eqref{eq:in_intersection} as a union of finitely many arithmetic progressions along with a set of Banach density equal to $0$. 
\end{remark}

Next we fix such an (infinite) arithmetic progressions $\{nk+\ell\}_{n\in\N_0}$ appearing in $\tilde{U}$ (for some given $k\in\N$ and $\ell\in\N_0$). Recalling the definition of the $\beta_n$'s as in Equation~\eqref{eq:def_beta_n} (see also equation~\eqref{eq:prop_beta_n}), we know that 
\begin{equation}
\label{eq:like_Gm}
\Phi^{nk+\ell}(\alpha)-\beta_{nk+\ell}\in  E^N\text{ for each }n\in\N_0.
\end{equation} 

\begin{lemma}
\label{lem:lrs_group}
Let $k$ and $\ell$ be as in Equation~\eqref{eq:like_Gm}. Then the  set of integers of the form $nk+\ell$ with the property that $\Phi^{nk+\ell}(\alpha)\in \Gamma_1$ is a union of finitely many arithmetic progressions along with a set of Banach density equal to $0$.
\end{lemma}

\begin{proof}
First, we recall (as shown in Lemma~\ref{lem:the_same}) that $\Phi^{nk+\ell}(\alpha)\in \Gamma_1$ if and only if $\Phi^{nk+\ell}(\alpha)-\beta_{nk+\ell}\in  H$. 

On the other hand, using Equation~\eqref{eq:like_Gm} along with \cite[Corollary~1.3]{BCE} and \cite[Corollary~1.3]{BG-2}  yields that for some generators $v_1,\dots, v_s$ of $E$ (and also at the expense of replacing the arithmetic progressions $\{nk+\ell\}$ by finitely many arithmetic progressions, which is again admissible due to Remark~\ref{rem:k-l}), we can write
\begin{equation}
\label{eq:Gm_0}
\Phi^{nk+\ell}(\alpha)-\beta_{nk+\ell}=\left( \prod_{i=1}^s v_i^{b^{(j)}_i(n)}\right)_{1\le j\le N},
\end{equation}
where each sequence $\{b_i^{(j)}(n)\}$ is a linear recurrence sequence of integers. Applying Proposition~\ref{ever 3} to the linear recurrence sequence  
$$\left(\prod_{i=1}^s  v_i^{b^{(j)}_i(n)}\right)_{1\le j\le N}$$
with respect to the subgroup $H$ of $E^N$ yields that the set of $n\in\N_0$ for which $\Phi^{nk+\ell}(\alpha)-\beta_{nk+\ell}\in H$ is a union of finitely many arithmetic progressions along with a finite set. 
\end{proof}

Lemma~\ref{lem:lrs_group} yields that the set $R_1=\tilde{R}_1$ (see also Lemma~\ref{lem:the_same}) is a finite union of arithmetic progressions along with a set of Banach density equal to $0$.


\subsection{Conclusion of the proof of Theorem \ref{thm:main} (i)}

We have shown that the set $R_1=R(G,\Phi,\alpha,\Gamma_1)$ is a union of finitely many arithmetic progressions along with a set of Banach density equal to $0$. 

Next, we let $\epsilon_0,\dots,\epsilon_{m-1}\in \bG_m^N(K)\subseteq G(K)$ be elements that are linearly independent with respect to $\Gamma_1$; the existence of such points $\epsilon_0,\dots, \epsilon_{m-1}$ is guaranteed by the fact that $$\dim_{\Q}\bG_m^N(K)\otimes_{\Z}\Q=\infty.$$ 
Then we let $\beta_i':=\beta_i+\epsilon_i$ for each $i=0,\dots, m-1$. We let $\Gamma_1'$ be the finitely generated subgroup spanned by $\Gamma$ along with $\beta_0',\dots, \beta_{m-1}'$. Then Proposition~\ref{prop:f.g.} yields that 
\begin{equation}
\label{eq:in_Gamma}
\Gamma_1\cap\Gamma_1'=\Gamma.
\end{equation} 
Applying the exact same argument as before to $R_1':=R(G,\Phi,\alpha,\Gamma_1')$, we conclude that $R_1'$ is also a union of finitely many arithmetic progressions along with a set of Banach density equal to $0$. But then using Equation~\eqref{eq:in_Gamma}, we get that $R=R(G,\Phi,\alpha,\Gamma)$ equals $R_1\cap R_1'$ and since intersections of finitely many arithmetic progressions is a union of finitely many arithmetic progressions, then we obtain the desired conclusion in Theorem~\ref{thm:main}.


\end{document}